\documentclass[
  12pt,
  a4paper,
  reqno,
  dvipsnames,
  svgnames,
  table
]{amsart}
\usepackage{amsmath,amssymb,amscd,amsthm,mathtools,mathrsfs}
\usepackage{dsfont}
\DeclareMathOperator{\Tor}{Tor}
 \newcommand{\KK}{\mathbb{K}}
  
\DeclareMathOperator{\Hom}{Hom}
\usepackage{bm}
\usepackage{graphicx}
\usepackage{tikz}
\usepackage{tikz-cd}
\usepackage{subcaption}
\usepackage{blkarray}
\usepackage{nicematrix}
\usepackage{booktabs}
\usepackage{url}
\usepackage{pifont}
\usepackage{xcolor}
\definecolor{GreenDef}{RGB}{27,132,5}
\definecolor{BlueDef}{RGB}{11,78,188}
\definecolor{RedDef}{RGB}{15,77,146}
\definecolor{Yblue}{RGB}{160,4,23}
\usepackage[top=2cm,bottom=2cm,left=1.9cm,right=1.9cm]{geometry}
\usepackage{titlesec}
\usepackage{titletoc}
\titleformat{\section}[hang]
{\scshape\small\color{RedDef}\filcenter}{\S\ \thesection.}{0.2em}{}
\titleformat{\subsection}[runin]
{\scshape\color{RedDef}}{\thesubsection.}{0.2em}{}[.]
\contentsmargin{2.55em}
\dottedcontents{section}[3.8em]{}{1.5em}{1pc}
\dottedcontents{subsection}[6.1em]{}{2.5em}{1pc}
\usepackage[breaklinks,colorlinks,pagebackref]{hyperref}
\definecolor{Red}{RGB}{160,4,23}
\definecolor{Green}{RGB}{27,132,5}
\definecolor{Blue}{RGB}{11,78,188}
\hypersetup{linkcolor=Red,citecolor=Green}
\usepackage[capitalise,nameinlink]{cleveref}
\crefformat{equation}{(#2#1#3)}
\usepackage[msc-links,lite,alphabetic]{amsrefs}
\usepackage{mathscinet}

\usepackage{ulem}
\def\MR#1{%
  \relax\ifhmode\unskip\spacefactor3000 \space\fi
  \href{http://www.ams.org/mathscinet-getitem?mr=#1}{MR#1}
}

\makeatletter
\renewcommand{\BibLabel}{%
  \Hy@raisedlink{\hyper@anchorstart{cite.\CurrentBib}\hyper@anchorend}%
  [\thebib]%
}
\makeatother
\theoremstyle{plain}
\newtheorem{theorem}{Theorem}[section]
\newtheorem{lemma}[theorem]{Lemma}
\newtheorem{proposition}[theorem]{Proposition}
\newtheorem{corollary}[theorem]{Corollary}

\theoremstyle{definition}
\newtheorem{definition}[theorem]{Definition}
\newtheorem{remark}[theorem]{Remark}
\newtheorem{example}[theorem]{Example}

\newcommand{\N}{\mathbb{N}}
\newcommand{\p}{\mathbb{P}}
\newcommand{\Z}{\mathbb{Z}}
\newcommand{\Q}{\mathbb{Q}}
\newcommand{\R}{\mathbb{R}}
\newcommand{\C}{\mathbb{C}}
\newcommand{\E}{\mathbb{E}}
\newcommand{\Aff}{\mathsf{Aff}}
\newcommand{\ASim}{\operatorname{Asim}}
\newcommand{\Sim}{\operatorname{Sim}}
 \newcommand{\GL}{\operatorname{GL}}
\newcommand{\Isom}{\operatorname{Isom}}
\newcommand{\lin}{\operatorname{lin}}
\newcommand{\supp}{\mathsf{supp}}

\newcommand{\Id}{\mathrm{Id}}

\title{Random free semigroups of affine groups}

\author[Richard Aoun]{Richard Aoun}
\address{University Gustave Eiffel, Champs-sur-Marne,
5 boulevard Descartes, 77420 Marne-la-Vall\'{e}e Cedex 2, France}
\email{richard.aoun@univ-eiffel.fr}
\author[Keivan Mallahi-Karai]{Keivan Mallahi-Karai}
\address{Constructor University, School of Science,
Campus Ring~I, 28759 Bremen}
\email{kmallahikarai@constructor.university}

\begin{document}
\maketitle
\begin{abstract}
We investigate random freeness of semigroups in the solvable non-virtually nilpotent setting. We focus on a correlated affine model, namely semigroups generated by the two components of a finitely supported random walk
$L_n=(L_{n,1},L_{n,2})$ on $\Aff(K)^2$ whose two components share a common linear part.

In this model, we show that the long-term behavior of freeness is completely governed by an abelian shadow, namely the projected random walk on the multiplicative group $A\subset K^\times$ generated by the common multipliers. If this walk is transient, then {the semigroup} $\langle L_{n,1},L_{n,2}\rangle_+$ is eventually free almost surely. If it is recurrent, then freeness does not stabilize: the set of non-free times is almost surely infinite, yet has almost sure density zero. Moreover, the   obstruction to freeness depends solely on the common linear part and admits an explicit arithmetic description in terms of  roots of Littlewood polynomials.

The proof combines a local contraction theorem for affine random walks over arbitrary local fields, developed in the appendix and related to the theory of critical affine random walks, with a ping--pong argument played at a time-dependent place.
\end{abstract}

  \tableofcontents

\section{Introduction}
The Tits alternative and its probabilistic variants have attracted considerable attention in geometric group theory over the past two decades. In many classes of non-amenable groups --- linear groups, hyperbolic groups, mapping class groups --- random walks generically generate free subgroups and freeness stabilizes almost surely. The present paper initiates the study of random freeness of semigroups in solvable non-virtually nilpotent groups. 

\medskip
\noindent\textbf{Deterministic background.}
The Tits alternative~\cite{Tits1972} asserts that every finitely generated
linear group either contains a free group on two generators or is virtually
solvable. On the solvable side, while   no non-abelian free 
subgroups can exist, Rosenblatt~\cite{Rosenblatt1974} showed that any finitely 
generated solvable non-virtually-nilpotent group nonetheless contains a 
free semigroup on two generators--- in particular, such groups have 
exponential growth, recovering a theorem of Milnor and 
Wolf~\cite{Milnor1968, Wolf1968}. 
Chou~\cite{Chou}
extended this dichotomy to the broader class of elementary amenable groups. 
 
Subsequent works of Alperin \cite{alperin} and Osin \cite{osin, osin2}  showed that finitely generated solvable groups of exponential growth  satisfy the stronger property of uniform exponential growth. Later, Breuillard~\cite{Breuillard} introduced a new approach to this phenomenon based on affine dynamics and ping--pong. Building on structural results of Groves~\cite{Groves}, he showed that, after passing to suitable finite-index subgroups and quotients, the problem reduces to subgroups of the affine group $\Aff(K)$ for a suitable field $K$. Besides yielding a new proof and strengthening of earlier results, this approach sheds new light on the mechanism underlying Rosenblatt's theorem by locating the source of freeness in affine dynamics. Using a similar reduction, Cornulier and Tessera~\cite{CornulierTessera} strengthened Rosenblatt's conclusion by proving that the free semigroup can in fact be chosen quasi-isometrically embedded. As the affine group is the smallest non-virtually-nilpotent solvable group, it is a natural model for studying similar phenomena pertaining to this class of groups.

\medskip

\noindent\textbf{Probabilistic Tits alternatives.}
The question of whether free generated subgroups  arise generically 
from random walks was first addressed by Guivarc'h~\cite{Guivarch1990}, 
who showed that along a subsequence of times, two independent random 
walks on a non-virtually-solvable linear group generate a free subgroup. 
The stronger statement --- that freeness holds \textit{eventually}, almost 
surely --- was established for finitely generated non-virtually-solvable 
linear groups in~\cite{Aoun2011} and for non-elementary hyperbolic 
groups in~\cite{GMO}, and has since been extended to acylindrically 
hyperbolic groups~\cite{TaylorTiozzo}, quantified via concentration 
inequalities~\cite{AounSert}, and carried over to circle 
diffeomorphisms~\cite{GilabertVio2024}. In all these cases, eventual freeness is driven by contraction or drift mechanisms intrinsic to the ambient group. 

\medskip
\noindent\textbf{The affine model.}
The analogous question for free semigroups in the solvable setting has not been addressed to our knowledge. In the non-amenable situations discussed above, the arguments ultimately rely on contraction or drift phenomena that are available for every adapted measure. In the solvable world, no such universal mechanism exists. The first solvable example that already displays this difference is the affine group $\Aff(K)$ of a field:
\[ \Aff(K):=\{ x \mapsto ax+b: a \in K^\ast, b \in K \} \]
 Note that $\Aff(K)$ carries a natural abelian quotient through its linear part map
$
a:\Aff(K)\to K^\times$, 
and. For a given random walk on $ \Aff(K)$, the induced random walk on this quotient may exhibit very different behaviors depending on the measure. 

In this paper, we study a model of random walks on $\Aff(K) \times \Aff(K)$ (see below for the definition) in order to prove an analogous result pertaining to the freeness of the semigroup generated by the components of the walk. Our result (Theorem \ref{thm.correlated-refined}) shows that recurrence properties of the projected random walk on $K^\times$ govern the asymptotic freeness of this random semigroup. 

 In our study of random walks    on $\Aff(K)$ for a global field $K$, two regimes arise naturally. The \textit{non-centered} case, where at least one place is contracting on average, and the \textit{centered} one, where  
no place contracts on average. The first case  presents similarities with the non-amenable world, even though the dynamics takes place on a non-compact space. In the centered case, however, for each local field embedding of $K$ the projected walk is individually in the \textit{critical regime}, meaning that the drift of the log-modulus of the linear part vanishes. One phenomenon that comes to help is a synchronization phenomenon for affine trajectories on local fields, originating in the work of Babillot--Bougerol--\'Elie~\cite{babillot.bougerol.elie}. The interplay of this phenomenon across different places is the key to our method.  

\medskip

\noindent\textbf{Results and structure of the paper.}\label{result}
Throughout the paper, $K$ is a field of an arbitrary characteristic and $\mu(K)$ denotes its group of roots of unity. For a subset $E$ of a group, we denote by
$
\langle E\rangle_{\mathrm{grp}}$
the subgroup generated by $E$, and by
$
\langle E\rangle_+$ 
the semigroup generated by $E$.
For $\gamma:x\mapsto a x+b\in \Aff(K)$, we write $a(\gamma):=a\in K^\times$ and refer to it as the linear part of $\gamma$.  
In this paper we only consider measures $\eta$ on 
$\Aff(K)^2$ that are finitely supported and whose support is included in the subgroup of $\Aff(K)^2$ consisting of pairs $(g_1,g_2)$ with $a(g_1)=a(g_2)$.  
For $g=(g_1,g_2)\in\langle\supp(\eta)\rangle_{\mathrm{grp}}$,   the common linear part of $g_1$ and $g_2$ is denoted by
\[
a(g):=a(g_1)=a(g_2).
\]
We denote by $a_\ast \eta$ the pushforward of $\eta$ by   the map $a:\Aff(K)\to K^\times$: $(a_\ast \eta)(x)=\sum_{a(g)=x}{\eta(g)}$, $x\in K^\times$. The case of more general measures (including the notable example of product measures) is the subject of a future work. 

 The random walk driven by $\eta$ is the process
\[
L_n=X_n\cdots X_1,
\]
where the $X_i$ are independent $\Aff(K)^2$-valued random variables with
common law $\eta$. We write  $L_n=(L_{n,1},L_{n,2})$ with $L_{n,i}\in \Aff(K)$. 
Let
\[
A=\langle \supp(a_\ast \eta) \rangle_
{\mathrm{grp}} \subset  K^\times.
\]
We consider the random walk $(a(L_n))$ on the finitely generated abelian group
$
A
$
driven by the probability measure $a_\ast \eta$, started at the identity
element $1\in K^\times$. We say that the walk is recurrent if
\[
\mathbb P(a(L_n)=1\ \mathrm{i.o.})=1,
\]
and transient otherwise.

Finally, let us define the notion of a free semigroup in a group. 
Let $G$ be a group and let $\alpha,\beta \in G$. We say that $\alpha$ and $\beta$ generate a free semigroup if any two distinct non-empty words in the letters $\alpha$ and $\beta$ always define distinct elements of $G$.

In particular, this is easily seen to be equivalent to requiring that, for all $k,l \geq 1$ and all distinct tuples
$
(m_1,n_1,\ldots,m_k,n_k)
\quad\text{and}\quad
(m'_1,n'_1,\ldots,m'_l,n'_l)
$
of positive integers, one has
$$
\alpha^{m_1}\beta^{n_1}\cdots \alpha^{m_k}\beta^{n_k}
\neq
\alpha^{m'_1}\beta^{n'_1}\cdots \alpha^{m'_l}\beta^{n'_l}.
$$

\medskip 
 One of the main results of this paper is the following dichotomy.  
\begin{theorem}[Asymptotic freeness and transience in the abelian quotient]
\label{thm.correlated-refined}
Let $K$ be a field of any characteristic, and let 
$\eta$ be a finitely supported probability measure on  $\Aff(K)^2$, supported on pairs $(g_1,g_2)$ with a common linear part
$a(g_1)=a(g_2)$. Assume the following non-degeneracy assumptions hold 
\begin{enumerate}
\item  The common linear part is not almost surely a root of unity, i.e.
$
\eta\{g:\ a(g)\in\mu(K)\}<1$. 
 
\item The support of $\eta$ is not included in the diagonal subgroup of $\Aff(K)^2$. 
\end{enumerate}
Then the following dichotomy holds: 
\begin{enumerate}
\item[i.] If the projected random walk $(a(L_n))$ on $A$ is transient, then
$\langle L_{n,1},L_{n,2}\rangle_+$ is eventually free almost surely.

\item[ii.] If the projected random walk $(a(L_n))$ on $A$ is recurrent, then
$\langle L_{n,1},L_{n,2}\rangle_+$ is not free infinitely often almost surely,
but
\[
\frac1N\#\{1\le n\le N:\langle L_{n,1},L_{n,2}\rangle_+
\text{ is not free}\}\underset{N\to +\infty}{\longrightarrow} 0
\]
almost surely.
\end{enumerate}
\end{theorem}

\begin{remark}
If condition (1) fails then, since $\mu(K)$ is a group,
$
a(g)\in\mu(K)$
for every
$
g=(g_1,g_2)\in\langle\supp(\eta)\rangle_{\mathrm{grp}}$. 
Hence, for every such $g$, there exists $m\in\N$ such that
$g_1^m$ and $g_2^m$ are translations, and therefore commute.
Thus $\langle g_1,g_2\rangle_+$ is never free. In particular,
$\langle L_{n,1},L_{n,2}\rangle_+$ is never free almost surely. A similar remark holds if condition (2) fails. 
\end{remark}
 
We will say that the probability measure $a_\ast \eta$  is centered if
\[\forall \chi\in \Hom(A,\R), \quad
\int \chi(t)\,d(a_*\eta)(t)=0.
\]
Equivalently, since every
homomorphism $A\to\R$ vanishes on the torsion group $\Tor(A)$, after identifying
$A/\Tor(A)$ with $\Z^r$, the induced probability measure on $\Z^r$
has zero mean vector. In particular, if $a_*\eta$ is symmetric, then it is
centered. 
\begin{corollary}[Rank/drift criterion]
\label{cor.intro-criterion}
Under the assumptions of Theorem \ref{thm.correlated-refined}, if $r$ denotes the  free  rank of  the finitely generated group $A$, then:
\begin{enumerate}
\item If $r\ge 3$, then almost surely there exists $n_0\in \N$ such that for every $n\geq n_0$, 
$\langle L_{n,1},L_{n,2}\rangle_+$ is free.
\item If $r\in\{1,2\}$, then 
$\langle L_{n,1}, L_{n,2}\rangle_+$ is a.s.~eventually free if and only if $a_\ast\eta$ is
not centered.
\end{enumerate}
\end{corollary}
An interesting feature of the corollary is that eventual freeness may still hold in the centered regime, where no place is contracting on average. In particular, when the rank of $A$ is $\geq 3$,  eventual freeness holds for any adapted measure, so that measure dependence is  a  low rank phenomenon.

\medskip
  \vspace{3mm}
\noindent\textit{Scheme of the proof.}
The proof has two main ingredients and a sharpening.

\medskip
\noindent
\noindent\textbf{Probabilistic part: eventual separation of fixed points.}
The key probabilistic step is to show that, almost surely, the fixed
points of $L_{n,1}$ and $L_{n,2}$ are eventually distinct whenever
$a(L_n)\neq 1$ (Proposition~\ref{prop.eventual-separation}). The
proof reduces, via a specialization argument
(Lemma~\ref{lemma.specialization}), to the case where $K$ is a
global field. The problem is a return to a subgroup problem, here the diagonal subgroup. 
We identify the dynamics of the random walk on $\Aff(K)^2$
modulo the diagonal subgroup with a one-dimensional affine recursion on $K$, thereby reducing the problem to that of showing that this process visits the origin only finitely often.
In the non-centered case, one place is contracting on average and
the affine recursion in the corresponding local field embedding has a
unique non-atomic stationary measure, ruling out an infinite number of returns to the origin.
The centered case is the heart of the
difficulty as  no place provides global contraction; and  therefore, by a theorem of Bougerol--Picard \cite{bougerol-picard}, no stationary probability measure exist for any local field embedding. The key input is a \textit{local contraction/synchronization
principle} for affine random walks  over local fields.  Roughly speaking this property says that  whenever one trajectory return to a given compact set, all trajectories asymptotically merge. This phenomenon
 was first established in the work of Babillot--Bougerol--\'Elie~\cite{babillot.bougerol.elie}
in the context of  affine random walks on the real line.   In \cite{brofferio-how},  Brofferio gave a new proof  of this result and, in \cite{brofferio.thesis, brofferio.renewal}, she proved that the same phenomenon happens for random walks on parabolic subgroups of the automorphism group of a homogeneous tree, a framework initiated by Cartwright-Kaimanovitch-Woess~\cite{CKW}. This applies, in particular,  for affine random walks on non-Archimedean local fields. 
 We develop in the appendix a unified version covering all local fields  
and extending them to the natural generality of  similarity transformations. Applied here, it forces $|a(L_n)|_v\to 0$
simultaneously at all relevant places on the event of infinitely many
returns, contradicting the product formula  
and completing the proof.

\medskip 
\noindent\textbf{Dynamical part: a moving ping--pong table.} 
Once fixed points are separated, freeness is established via a ping--pong argument of a new kind. In classical settings, a single place provides uniform contraction and the ping--pong table is fixed. Here, especially in the centered case, no place plays this role uniformly. Instead, whenever the common multiplier is sufficiently far from the identity in the logarithmic embedding of $A$ one can find a place, thanks to the product formula,   where it is strongly contracting. The relevant place depends on the position of the walk and therefore changes with time --- hence the name \textit{moving ping--pong table}.
This construction shows that non-freeness can only occur when the multiplier remains in a bounded arithmetic region. Dirichlet's $S$-unit theorem implies that this region intersects $A$ in a finite set. In the transient regime, the projected walk eventually leaves this set forever, yielding eventual freeness. In the recurrent regime, the walk returns to the identity infinitely often, producing infinitely many non-free times, but visits the obstruction set with density zero. This yields density-one freeness.

\medskip
\noindent\textbf{Arithmetic sharpening: Littlewood polynomials.}
The dynamical argument above identifies the obstruction to freeness as a finite subset of $A$.
A careful analysis reveals that the exact obstruction admits an explicit arithmetic description in terms of roots of unity and roots of Littlewood polynomials (Remark~\ref{rem.littlewood.sharp}), giving a sharp characterization of non-free times that goes beyond what the dynamical argument detects.

 \section{Arithmetic and probabilistic preliminaries}
\label{sec.prelim}
  
\subsection{Random walks on finitely generated abelian groups}\label{sec.rw-abelian}
In this section, we study recurrence and transience properties of
random walks on finitely generated abelian groups and relate them to
centeredness (Proposition \ref{prop.rw-abelian}). 

Let $B$ be an abelian group and let $\zeta$ be a finitely supported probability
measure on $B$. Set
\[
A:=\langle\supp(\zeta)\rangle_{\mathrm{grp}},
\qquad
A^+:=\langle\supp(\zeta)\rangle_+ .
\]
Thus $A$ is a finitely generated abelian group. We consider the random walk
\[
W_n=X_1\cdots X_n,\qquad W_0=e,
\]
where the $X_i$ are independent $B$-valued random variables with common law
$\zeta$. 
 
The random walk
$(W_n)$
defines a Markov chain on the countable  state space $A$ with transition kernel
$
P(g,h)=\zeta(g^{-1}h)$. The accessible states from $e$ are precisely the elements of $A^+$.  More generally, the states   accessible from
 $x\in A$ are 
 $xA^+$.  
  \begin{definition}
We say that the random walk $(W_n)$ is \textit{recurrent} 
\[
\mathbb P(W_n=e \ \mathrm{i.o.})=1.
\]
Otherwise, we say that the walk is \textit{transient}.
\end{definition}


The recurrence criterion below is governed by the following notion of
centeredness.  \begin{definition}\label{def.center}
We say that $\zeta$ is centered if
\[
\int \chi(a)\,d\zeta(a)=0
\]
for every group homomorphism $\chi:A\to\R$.
\end{definition}

\begin{lemma}\label{lem.centered-zr}
The measure $\zeta$ is centered if and only if, after identifying
$A/\Tor(A)$ with $\Z^r$, the induced probability measure on $\Z^r$
has zero mean vector.
\end{lemma}
\begin{proof}
Since $\R$ is torsion-free, every homomorphism $\chi:A\to\R$
vanishes on $\Tor(A)$ and therefore factors through
\[
A/\Tor(A)\simeq\Z^r.
\]
Thus
\[
\Hom(A,\R)\simeq \Hom(\Z^r,\R)\simeq (\R^r)^\ast.
\]
Hence centeredness is equivalent to vanishing of the mean vector against
every linear form on $\R^r$, i.e. to the mean vector being zero.
\end{proof}
We start with the following observation. 
\begin{lemma}\label{lem.centered-semigroup-group}
If $\zeta$ is centered, then
\[
A^+=A.
\]
In particular,  the Markov chain  $(W_n)$ is irreducible on
$A$. 
 \end{lemma}

\begin{proof}
Let
$
S=\supp(\zeta)
$
and write the group law additively in
$
A_0:=A/\Tor(A)\simeq \mathbb Z^r$. 
Since $\zeta$ is centered, the origin belongs to the convex hull of the image
$\overline S$ of $S$ in $A_0$. Hence there exist real numbers
$(p_s)_{s\in S}$ such that
\[
\sum_{s\in S} p_s\,\overline s =0,
\qquad
p_s>0,
\qquad
\sum_{s\in S}p_s=1.
\]
Since the above equations form a linear system with rational coefficients, its solution set is a   subspace of $\R^{|S|}$ defined over $\Q$. Its rational points form a dense subset of this subspace. Since $(p_s)_{s\in S}$ is a solution in the positive orthant of $\R^{|S|}$, we may find a rational solution $(q_s)_{s\in S}$ with $q_s>0$ for every $s\in S$.  Multiplying by a common denominator, we obtain positive
integers $(n_s)_{s\in S}$ such that
\[
\sum_{s\in S} n_s\,\overline s =0
\quad\text{in } A_0.
\]
Equivalently, in multiplicative notation in $A$,
\[
t:=\prod_{s\in S}s^{n_s}\in\Tor(A).
\]
Let $m\geq1$ be such that $t^m=e$. Then
\[
\prod_{s\in S}s^{mn_s}=e.
\]
Fix $s_0\in S$. Rearranging this relation gives
\[
s_0^{-1}
=
s_0^{mn_{s_0}-1}
\prod_{\substack{s\in S\\ s\neq s_0}}s^{mn_s}.
\]
The right-hand side belongs to $A^+$, hence
$s_0^{-1}\in A^+$.
Since this holds for every $s_0\in S$, the semigroup $A^+$ contains the
inverse of each of its generators. Therefore
$
A\subset A^+$. 
Since the reverse inclusion is immediate, we conclude that
$
A^+=A$.  This implies the irreducibility of the Markov chain on $A$, since for every $x\in A$, $xA^+=xA=A$. 
 \end{proof}

We can therefore apply classical recurrence  criteria for
irreducible random walks on $\Z^r$.
\begin{proposition}\label{prop.rw-abelian}
Let $r$ be the free   rank of $A$. Then 
\begin{enumerate}
\item[(i)]
If $\zeta$ is not centered, then $(W_n)$ is transient. 

\item[(ii)]
If $\zeta$ is centered, then $(W_n)$ is recurrent if and only if $r\leq2$.
\end{enumerate}
Moreover, in the transient regime, for any $x\in A$ and for any $F\subset A$ finite, $\p_x(W_n\in F\,\textrm{i.o.})=0$. 
\end{proposition}
\begin{proof}
 Write $$
A\simeq \mathbb Z^r\times T,
$$
where $T$ is finite.

Assume first that $\zeta$ is not centered. Then the induced walk $\overline{W_n}$ on
$A/\Tor(A)\simeq\mathbb Z^r$ has nonzero mean.   By the law
of large numbers,
$
\frac{\overline{W_n}}{n}\to m\neq0
$
in $\mathbb R^r$, hence $(\overline{W_n})$ escapes to infinity. Since $W_n=e$ implies  $\overline{W}_n=0$, we deduce the transience of $(W_n)$.  

Assume now that $\zeta$ is centered. By
Lemma~\ref{lem.centered-semigroup-group},
$A^+=A$, 
and the walk $W_n$ is irreducible on $A$ and {\it a fortiori} the projected random walk $\overline{W_n}$ is irreducible on  $\Z^r$. 
Hence, using classical results 
(see \cite[Chapter~3]{woess} or \cite[Chapter~2]{spitzer}), we deduce that $\overline{W}_n$ is recurrent if and only if $r\leq 2$. If $r\leq 2$, $\overline{W_n}$ returns to $0$ infinitely often.  Note that $ \overline{W_n}=0$ is equivalent to $W_n \in T$. Since $T$ is finite, there must exist a point in $T$ which is visited infinitely often. By irreducibility of $(W_n)$ this implies that $e$ is visited infinitely often and recurrence follows. If $r\geq 3$, transience of $\overline{W_n}$ implies that of $W_n$. 

It remains to prove the last assertion. If $\zeta$ is not centered, then again by the law of large numbers, $\overline{W_n}$ eventually escapes  any finite set and hence $W_n x$ eventually escapes any finite set. 
If $\zeta$ is centered and the walk is transient (which happens when $r\geq 3$), then by
Lemma~\ref{lem.centered-semigroup-group} the chain is irreducible on $A$.
Transience of one state therefore implies transience of every state, and hence
every finite set is visited only finitely often almost surely. 
\end{proof}
  \subsection{Arithmetic characterization of centeredness}
In this subsection, we give an arithmetic characterization of centered
probability measures on finitely generated subgroups of $F^\times$, where $F$ is
a global field. More precisely, we show that centeredness can be detected by
vanishing of logarithmic drifts at all places of $F$
(Lemma~\ref{lemma.centeredness-places}).

\medskip 
In all this section, $F$ is a global field, i.e.~a number field (finite extension of $\Q$) in characteristic zero and a finite extension of $\mathbf{F}_q(T)$ for some finite field $\mathbf{F}_q$ in positive characteristic. A reference is \cite[Ch I]{neukirch} for the number field case and \cite[Ch 5]{Rosen} for the function field case. Both classes of  fields arise naturally in affine models from geometric group theory: for instance co-compact lattices $\Z\ltimes_A \Z^2$ with $A\in SL_2(\Z)$  hyperbolic  in the SOL group  embed in $\Aff(K)$ with $K=\Q(\sqrt {\operatorname{tr}(A)^2-4})$, and  lamplighter groups $\Z\wr (\Z/p\Z)$ 
embed in $\Aff(\mathbf{F}_p(T))$.

\medskip 
\noindent\textit{Places of global fields}.

For every non-archimedean place $v$ of $F$ we denote by     
 $v: F^\times\to \Z$ the corresponding normalized discrete valuation. Let $k(v)$ denote the  residue field and set $q_v:=|k(v)|$ its cardinality.  For $x\in F^\times$, its absolute value is 

\begin{equation}\label{eq.absolute.value}|x|_v=q_v^{-v(x)}.\end{equation}

The archimedean places correspond to field embeddings
$
\sigma:F\hookrightarrow\mathbb R$ or 
$\sigma:F\hookrightarrow\mathbb C$. 
They define archimedean absolute values by
\[
|x|_\sigma :=
\begin{cases}
|\sigma(x)| & \text{if }\sigma:F\hookrightarrow\mathbb R,\\
|\sigma(x)|^2 & \text{if }\sigma:F\hookrightarrow\mathbb C.
\end{cases}
\]
For every $x\in F^\times$, one has $|x|_v=1$ for all but finitely many places $v$. 
With these normalizations, the family of absolute values
$(|\cdot|_v)_v$, indexed by all places of $F$, satisfies the product formula
\[
\prod_v |x|_v = 1
\qquad\text{for every }x\in F^\times.
\]
Equivalently,
\begin{equation}
\label{eq.product.formula}
\sum_v \log |x|_v = 0.
\end{equation}

Note that in positive characteristic, all places are non-archimedean ones. 

\begin{example}

\begin{enumerate}

    \item 
For $F=\Q$, the non-archimedean places correspond to the prime numbers $p$, with associated valuation $v_p$ and absolute value
$$
|x|_p=p^{-v_p(x)}.
$$
The unique archimedean place is the usual absolute value on $\Q$. 

\item If $F$ is a number field   and $\mathcal{O}_F$ its ring of integers. 
 Non-archimedean places are in one-to-one correspondence with nonzero prime ideals of $\mathcal{O}_F$. The ring 
$\mathcal{O}_F$ is a Dedekind domain and every nonzero fractional ideal factors uniquely as
a product of prime ideals. For every nonzero prime ideal
$\mathfrak p\subset \mathcal{O}_F$, the associated valuation $v_{\mathfrak p}$
is characterized by
$$
(x)=\prod_{\mathfrak p}\mathfrak p^{\,v_{\mathfrak p}(x)},
$$
for $x\in F^\times$, where all but finitely many exponents vanish. The
corresponding absolute value is
$$
|x|_{\mathfrak p}=
q_{
\mathfrak{p}}^{-v_{\mathfrak p}(x)}.
$$
The archimedean places correspond to the embeddings
$$
\sigma:F\hookrightarrow \R
\qquad\textrm{and}\qquad
\sigma:F\hookrightarrow \C.
$$

\item If $F=\mathbf F_q(T)$, the non-archimedean places correspond to the
irreducible polynomials $P(T)\in\mathbf F_q[T]$ together with the place at
infinity. The associated valuations are given by
\[
v_P\!\left(P^n\frac{f}{g}\right)=n,
\]
whenever $f,g\in\mathbf F_q[T]$ are not divisible by $P$, and
\[
v_\infty\!\left(\frac{f}{g}\right)
=
\deg(g)-\deg(f).
\]

\end{enumerate}
\end{example}

 
\medskip
\noindent\textit{Logarithmic embedding of finitely generated subgroups.}
Let $A\le F^\times$ be a finitely generated subgroup and define the finite set
of relevant places
\begin{equation}\label{eq.places.relevant}
S(A):=\{v:\exists a\in A,\ |a|_v\neq1\}.\end{equation}
Since $A$ is finitely generated, the set $S(A)$ is finite.
 Hence the product formula reduces to
\[
\sum_{v\in S(A)}\log|a|_v=0
\]
for every $a\in A$.
Define the logarithmic map
\[
\lambda:A\longrightarrow\R^{S(A)},
\qquad
\lambda(a):=(\log|a|_v)_{v\in S(A)}.
\]
The following lemma shows that $\lambda$ detects precisely the non-torsion part
of $A$.
\begin{lemma}[Kernel of the logarithmic embedding]\label{lemma.ker-lambda}
One has 
\[
\ker(\lambda)=\Tor(A).
\]
In particular, $\lambda$ induces an injective homomorphism
\[
\overline{\lambda}:A/\Tor(A)\hookrightarrow \R^{S(A)}.
\]
The map $\lambda$ is moreover proper: inverse images of compact subsets of $\R^{S(A)}$ are finite. 
\end{lemma}

\begin{proof}
If $a\in\Tor(A)$, say $a^m=1$, then for every place $v$ one has $|a|_v^m=|a^m|_v=1$, hence    $|a|_v=1$ and therefore   $\lambda(a)=0$. Thus $\Tor(A)\subset\ker(\lambda)$.

Conversely, let $a\in A$ with $\lambda(a)=0$.  The assumption on $a$ and definition of $S(A)$ imply that $|a|_v=1$ for every place $v$ of $F$.
Suppose first $\textrm{char}(F)=0$. Since $v_{\mathfrak{p}}(a)=0$ for every prime ideal $\mathfrak{p}$ of the ring of integers $\mathcal{O}_F$, then $a$ is an algebraic unit. Using in addition that $|\sigma(a)|=1$ for every embedding $\sigma:F\to \C$,     Kronecker's theorem yields that $a$ is a root of unity. If $F$ has positive characteristic,   then  by    \cite[Proposition 5.1]{Rosen} $a$ belongs to the constant field $\mathbf{F}_q$, and since $\mathbf{F}_q^\times$ is finite,     $a$ is a root of unity. In both cases $a\in \Tor(A)$ so $\ker(\lambda)=\Tor(A)$. This  yields an injective group homomorphism 
$\overline{\lambda}: A/\Tor(A)\simeq \Z^r\to \R^{S(A)}$. 

 Finally, we show properness.
 Suppose first $\textrm{char}(F)=0$. By definition, $A\subset O_{F,S(A)}^\ast$, where 
$\mathcal O_{F,S(A)}^\ast=\{x\in F^\times; |x|_v=1, \forall v\not\in S(A)\}$  is the group of $S(A)$-units of $K$. It follows from Dirichlet's  $S$-unit  theorem  (see for example \cite[Theorem 3.12, Lemma 3.14]{Narkiewicz}) that $\lambda(\mathcal O_{F,S(A)}^\ast)$ is a lattice in the hyperplane
$
 \sum_{v\in S(A)}x_v=0$ of $\R^{S(A)}$,  
 hence discrete in $\R^{S(A)}$.  In particular, its  subgroup
$
\overline{\lambda}(A/\Tor(A))
$
is discrete in
$\R^{S(A)}$.
Therefore every compact subset $C \subseteq \R^{S(A)}$ intersects
$\overline{\lambda}(A/\Tor(A))$ in a finite set.
Since $\overline{\lambda}$ is injective, it follows that
$\overline{\lambda}^{-1}(C)$ is finite.
Hence $\overline{\lambda}$ is proper in  zero characteristic. In positive characteristic, the situation is simpler: since all places are non-archimedean, it follows from \eqref{eq.absolute.value} that $\lambda(A)\subset \prod_{v\in S(A)}{(\log q_v) \Z}$ which is clearly  discrete yielding properness of $\overline{\lambda}$ by injectivity of 
$\overline{\lambda}$ as argued in the zero characteristic case.
\end{proof}
\medskip
\noindent\textit{Characterization of centered measures using logarithmic drifts.}

\begin{lemma}[Detection of centeredness by places]\label{lemma.centeredness-places}
Let $\zeta$ be a probability measure on $F^\times$ and denote by $A:=\langle \supp \zeta \rangle_{\mathrm{grp}}$. Then the following are equivalent: 
\begin{enumerate}
\item[(i)] $\zeta$ is centered.
\item[(ii)] For every place $v\in S(A)$,
\[
\int \log|a|_v\,d\zeta(a)=0.
\]
\item[(iii)] After identifying $A/\Tor(A)$ with  $\Z^r$, the induced 
distribution on  $\Z^r$ has zero mean.
\end{enumerate}
\end{lemma}

\begin{proof}
The equivalence $(i)\Longleftrightarrow(iii)$ is general and follows from Lemma \ref{lem.centered-zr}. \\

\textit{(i)$\Rightarrow$(ii).}
For each $v\in S(A)$, the map $a\mapsto\log|a|_v$ is a group homomorphism
$A\to\R$. Centeredness therefore implies
\[
\int\log|a|_v\,d\zeta(a)=0.
\]
\medskip
\textit{(ii)$\Rightarrow$(i).}
Let $\pi:A\to A_0:=A/\Tor(A)$ be the canonical projection and
let $\bar\lambda:A_0\to\R^{S(A)}$ be the homomorphism induced by $\lambda$.
By Lemma~\ref{lemma.ker-lambda}, $\bar\lambda$ is injective.
Tensoring with $\R$, we obtain an injective linear map
\[
\bar\lambda_\R : A_0\otimes_\Z \R \longrightarrow \R^{S(A)}.
\]
Since $A_0\simeq\Z^r$, we identify $A_0\otimes_\Z \R\simeq\R^r$,
and $\bar\lambda_\R$ is an isomorphism onto its image
\[
V:=\mathrm{Im}(\bar\lambda_\R)\subset \R^{S(A)}.
\]
Let $\chi:A\to\R$ be a group homomorphism.
Since $\R$ is torsion-free, $\chi$ vanishes on $\Tor(A)$ and therefore factors
through $A_0$ i.e.     there exists a unique homomorphism
\[
\chi_0:A_0\to\R
\]
such that $\chi=\chi_0\circ\pi$.
Extending scalars, $\chi_0$ induces a linear functional
\[
\chi_{0,\R}: A_0\otimes_\Z \R \longrightarrow \R.
\]

Define
\[
\ell := \chi_{0,\R} \circ \bar\lambda_\R^{-1} : V\to\R.
\]
This is a linear form on $V$. Extend $\ell$ to a linear functional on
$\R^{S(A)}$.
Then there exist real numbers $(c_v)_{v\in S(A)}$ such that
\[
\forall  x=(x_v)_{v\in S(A)}, \,\,
\ell(x)=\sum_{v\in S(A)} c_v\, x_v.
\]
For $a\in A$ we compute:
\[
\chi(a)
=
\chi_0(\pi(a))
=
\ell\big(\bar\lambda(\pi(a))\big)
=
\ell\big(\lambda(a)\big)
=
\sum_{v\in S(A)} c_v\,\log|a|_v.
\]
Integrating and using~(ii) gives
\[
\int \chi(a)\,d\zeta(a)=0.
\]
Since $\chi$ was arbitrary, $\zeta$ is centered.
\end{proof}

\begin{example}
Consider a probability measure $\eta$ on $\Aff(\Q)$ supported on
$$
x\mapsto 12 x,\qquad
x\mapsto \frac{3}{5} x+1,\qquad
x\mapsto \frac{5}{2} x,\qquad
x\mapsto \frac{1}{3} x.
$$
The subgroup generated by the linear parts satisfies 
$$
A=\left\langle 12,\frac{3}{5},\frac{5}{2}, \frac{1}{3}\right\rangle
=\langle 2,3,5\rangle \simeq \Z^3.
$$
The relevant places are $S(A)=\{2,3,5,\infty\}$.
By the product formula, the condition at $\infty$ is a linear consequence
of the non-archimedean ones, so centeredness is equivalent to the vanishing
of the average valuation vector in $\Z^3$. Using the coordinates 
$
a\longmapsto (v_2(a),v_3(a),v_5(a))
$
on $A=\langle 2,3,5\rangle\simeq\Z^3$,  the corresponding valuation vectors are
\[
(2,1,0),\qquad
(0,1,-1),\qquad
(-1,0,1),\qquad
(0,-1,0).
\]
If the above affine maps are chosen with probabilities
$p_1,p_2,p_3,p_4$, then by Lemma~\ref{lemma.centeredness-places},
the projected measure on $A$ is centered if and only if
$$
p_1(2,1,0)
+p_2(0,1,-1)
+p_3(-1,0,1)
+p_4(0,-1,0)
=(0,0,0).
$$
This yields the unique solution
\[
p_1=\frac{1}{8},\qquad
p_2=\frac{1}{4},\qquad
p_3=\frac{1}{4},\qquad
p_4=\frac{3}{8}.
\]
\end{example}

\section{Proof of Theorem    \ref{thm.correlated-refined}}

\subsection{Eventual separation of fixed points}
Throughout this section, we assume that $K$ is a {global field} and that $\eta$ satisfies the assumptions of
Theorem~\ref{thm.correlated-refined}. In the sequel, we denote the element of $\Aff(K)$ mapping $x$ to $ax+b$ by $(a,b)$. In this notation, the group law is written as 
$(a_1, b_1) (a_2, b_2)= (a_1a_2, a_1b_2+b_1).$ Moreover, $(a,b)^{-1}=(a^{-1},-ba^{-1})$. 
\begin{proposition}
  [Eventual separation]
\label{prop.eventual-separation}
Let $K$ be a {global} field. Under the assumptions of Theorem~\ref{thm.correlated-refined}, almost surely there
exists $n_0\in\N$ such that for all $n\ge n_0$,
\[
a(L_n)\neq 1 \Longrightarrow L_{n,1}^+\neq L_{n,2}^+.
\]
\end{proposition}\smallskip
\smallskip
 
\begin{proof}
\textit{An auxiliary affine recursion.} 
Note that since $a(g_1)=a(g_2)=:a(g)$ for every $g=(g_1,g_2)\in G_\eta$, when $a(g)\neq 1$, $g_1^+=g_2^+$ if and only if $g_1=g_2$. Hence the proposition is a return to the diagonal subgroup   
$
\Delta:=\{(g,g): g\in\Aff(K)\}\subset\Aff(K)^2$ problem. Set
$
H:=G_\eta\cap\Delta$. 
Then
\[
L_{n,1}=L_{n,2}
\Longleftrightarrow
L_n\in H.
\]
Let $\theta: G_\eta \to \Aff(K),$ defined for every $g=(g_1,g_2)\in G_\eta$ by,   
\[
\theta(g)(x):=a(g)x+b(g_2)-b(g_1).
\]
Since $\delta(g):=b(g_2)-b(g_1)$ satisfies 
  the  cocycle  relation $\delta(gh)=a(g)\delta(h)+\delta(g)$ (being the difference of two cocycles $g\mapsto b(g_1)$ and $g\mapsto b(g_2)$), $\theta$ is a group homomorphism. 
  This defines an  affine action of $G_\eta$ on $K$, and we denote 
 $g\cdot x$ instead of $\theta(g)(x)$. Moreover, its stabilizer at $0$ is precisely $H$. Therefore
the quotient space $G_\eta/H$ identifies with the orbit
\[
\Omega:=G_\eta\cdot 0\subset K.
\]
The process
\[
Y_n:=L_n\cdot 0
\]
is thus a Markov chain on the countable set $\Omega$, and
\[
L_n\in H
\Longleftrightarrow
Y_n=0.
\]
Equivalently,
\[
L_{n,1}=L_{n,2}
\Longleftrightarrow
Y_n=0.
\]
Hence, our goal is to show that $\p(Y_n=0\,\textrm{i.o.})=0$.

\medskip 
\textit{Case A. Quotient action has a common fixed point.}
Let $G_\eta:=\langle \supp \eta \rangle_{\mathrm{grp}}$. 
Assume first that there exists $x_0\in K$ such that for every $g=(g_1,g_2)\in G_\eta$, $(g_1,g_2)\cdot x_0=x_0$ i.e. $a(g) x_0+b(g_2)-b(g_1)=x_0$. 
Since by assumption $G_\eta\not\subset \Delta$, one has $x_0\neq 0$. Let $\tau=(1,x_0)$ be the translation by $x_0$. 
For every $g=(g_1,g_2)\in G_\eta$, one has $$\tau^{-1} g_2\tau=(a(g),a(g) x_0+b(g_2) -x_0)=(a(g), b(g_1))=g_1.$$
Hence 
$$G_\eta \subset \{(g, \tau g \tau^{-1}); g\in \Aff(K)\}.$$
Now let   $(g_1,g_2)\in G_\eta$ with  $a(g_1)\neq 1$. Then $g_1$ has a unique fixed point $x^+$, and since $g_2=\tau g_1\tau^{-1}$, 
the unique fixed point of $g_2$ is $\tau(x^+)=x^++x_0\neq  x^+$ because $x_0\neq 0$. 
This finishes the proof  in this case.

\medskip 
\textit{Case B. Quotient action has no common fixed point.}

\medskip
\noindent\textit{Case B.1: $a_\ast\eta$ is not centered.}
By Lemma~\ref{lemma.centeredness-places} and the product formula \eqref{eq.product.formula},  there exists a place $v\in S(A)$ such that
\[
\E\big[\log|a(X_1)|_v\big]> 0.
\]
Consider the   affine random walk on 
$K_v$ induced by $\theta_\ast \check{\eta}$, where $\check{\eta}$ is the law of $X_1^{-1}$. We have   $\E[\log |a(X_1^{-1})|_v]<0$. 
Denote $\check{R}_n=X_1^{-1}\cdots X_n^{-1}$ the right random walk on $\Aff(K_v)$ induced by $\theta_\ast \check{\eta}$. 
By the standard theory of contracting (in average) affine recursions on local fields, the process $\check{R}_n\cdot 0$ converges almost surely to a random variable $Z_\infty$ whose  distribution $\nu$ is the unique $\theta_\ast \check{\eta}$-stationary probability measure on $K_v$\footnote{This is a standard fact; we include a proof for the convenience of the reader. 
Set $\mu:=\theta_\ast\check\eta$ and $Y_i:=X_i^{-1}=(a_i',b_i')$, so that
$\check R_n\cdot 0=Y_1\cdots Y_n\cdot 0=\sum_{k=0}^{n-1}a_1'\cdots a_k'\,b_{k+1}'$.
This series converges absolutely almost surely: since $\E[\log|a_1'|_v]<0$,
the law of large numbers gives $|a_1'\cdots a_k'|_v\to 0$ exponentially fast,
while $\log^+|b_k'|_v=o(k)$ a.s. (the variable $\log^+|b'|_v$ being
integrable). Its a.s. limit $Z_\infty$ satisfies
$Z_\infty=a_1'\,(Z_\infty\circ\theta)+b_1'$, where $\theta$ is the shift. 
Since $Z_\infty\circ\theta$ has the same law as $Z_\infty$ and is independent
of $Y_1$, the law $\nu$ of $Z_\infty$ is $\mu$-stationary, i.e.~
$\sum_{g}\mu(g)\,\nu(g^{-1}\cdot x)=\nu(x)$ for every $x\in K_v$.}.  
Moreover, Assumption of Case B implies that  the affine maps in the support of $\theta_\ast \eta$  
have no common fixed point in $K_v$. Indeed, a common fixed point $x$ satisfies
$
x=\frac{b(g_2)-b(g_1)}{1-a(g)}$
for every $g\in G_\eta$ with $a(g)\neq 1$ (such an element exists by assumption (i) of the theorem), hence $x$ necessarily belongs to $K$. Therefore the same holds for $\supp \theta_\ast \check{\eta}$. 
A classical argument due to Furstenberg shows that $\nu$ is non-atomic\footnote{
 Let $r:=\max\{\nu(\{x\}):x\in K_v\}$ and suppose,
for contradiction, that $r>0$; then $O:=\{x\in K_v:\nu(\{x\})=r\}$ is finite
and nonempty. For $x\in O$ and $g\in\supp\mu$, stationarity together with the
maximality of $r$ forces $\nu(\{g^{-1}\cdot x\})=r$, that is
$g^{-1}\cdot x\in O$; hence $g^{-1}(O)\subseteq O$. As $O$ is finite and each
$g^{-1}$ is injective, this inclusion is an equality $g^{-1}(O)=O$, so every
$g\in\supp\mu$ permutes $O$, and therefore so does the whole group
$\langle\supp\mu\rangle_{\mathrm{grp}}$. The set $O$ must be  a singleton; indeed  fix $g=(a,b)\in \supp \mu$  with $a\in K_v$ not a root of unity (thanks to assumption (1)) and denote by $x_g^+$ its unique fixed point, then the relation $g^n x= x_g^++a^n(x-x_g^+)$ true for every $x\in K_v$ shows that the orbit of any $x\neq x_g^+$ is infinite (as $a$ is not a root of unity) so necessarily $O=\{x_g^+\}$.    Finally
$\langle\supp\mu\rangle_{\mathrm{grp}}
=\langle(\supp\theta_\ast\eta)^{-1}\rangle_{\mathrm{grp}}
=\langle\supp\theta_\ast\eta\rangle_{\mathrm{grp}}$, so this group has a common
fixed point in $K_v$, contradicting the no-common-fixed-point hypothesis of
Case~B. Hence $r=0$ and $\nu$ is non-atomic. 
}.
In particular $\nu(\{0\})=0$ so that a.s. $Z_\infty\neq 0$. Since $\check{R_n}\cdot 0\to Z_\infty$ a.s.  we get that  a.s. eventually one has $\check{R}_n\cdot 0\neq 0$. 
Using $\check{R}_n=L_n^{-1}$  we get also a.s. eventually $L_n\cdot 0\neq 0$, which is what we want to show.

\medskip
\noindent\textit{Case B.2: $a_*\eta$ is centered.}
Let
\[
E:=\{Y_n=0, a(L_n)\neq 1\, \text{ infinitely often}\}.
\]
Arguing by contradiction, we
  assume $\p(E)>0$.
On $E$, there exists an infinite random subsequence $(n_k)_{k\in\N}$ with
$Y_{n_k}=0$ for all $k$. 

Recall that $S(A)$  denotes the finite set of places $v$  such that
$
|a|_v\neq1$
for some $a\in A$, see \eqref{eq.places.relevant}. We have $S(A)\neq \emptyset$ as otherwise, by Kronecker's theorem, 
$A$ would be included in the group of roots of unity $\mu(K)$ of $K$,  contradicting assumption (1) of Theorem \ref{thm.correlated-refined}.

Fix $v\in S(A)$.
The linear part of the affine random walk induced by $\theta_\ast \eta$ on $K_v$ has zero drift. We would like to apply the local contraction phenomenon for it. 
We now   check the assumptions Theorem \ref{thm.local-contraction} of the appendix for the probability measure $\theta_\ast \eta$ on $\Aff(K_v)$ with the local field $K_v$ and $V:=K_v$. Clearly $G_{\theta_\ast \eta}$ is a subgroup of similarities of the one-dimensional space $V=K_v$ over the local field $K_v$. The moment condition (i) of the aforementioned theorem is immediate as $\supp \theta_\ast \eta$ is finite. 
The affine non-degeneracy (ii) follows from the assumption of Case B together with the fact that fixed points must belong to the base field $K$. The zero-drift condition (iii)  follows from the  criticality assumption on $a_\ast \eta$.  
Now we check (iv). Since $v\in S(A)$, there exists $a\in A$ such that $|a|_v\neq 1$. 
Since $A$ is generated by $a(\supp\eta)$, not all elements of
$a(\supp\eta)$ can satisfy $|a(g)|_v=1$. Hence
$
\eta(\{g;\ |a(g)|_v=1\})<1$, 
so assumption (iv) is fulfilled.
We may   therefore apply   Theorem \ref{thm.local-contraction}  with the starting point $0$ and the compact $\{0\}$, it yields that    
\[
|a(L_n)|_v\,\mathbf 1_{\{Y_n=0\}} \longrightarrow 0
\quad\text{almost surely}.
\]

Setting 
\[
F:=\bigcap_{v\in S(A)}
\left\{
|a(L_n)|_v\,\mathbf 1_{\{Y_n=0\}}\longrightarrow 0
\right\}.
\]
we deduce that $\p(F)=1$, so $\p(E\cap F)>0$.
On $E\cap F$, we have $Y_{n_k}=0$ for every $k$, hence
\[
|a(L_{n_k})|_v \longrightarrow 0
\quad\text{for all } v\in S(A).
\]
 For $v\notin S(A)$, by definition of $S(A)$,  we have $|a(L_{n_k})|_v=1$. 
Therefore,
\[
\prod_{v} |a(L_{n_k})|_v \longrightarrow 0,
\]
contradicting the product formula
\eqref{eq.product.formula}. 
This completes the proof.
\end{proof}

  \subsection{Proof of the main theorem}\label{sec.proof}
 
We begin with a specialization lemma reducing the dynamical arguments of the proof of the theorem to the
number field setting. 

\begin{lemma}[Admissible specialization]
\label{lemma.specialization}
Let $K$ be a field and let $\eta$ be a finitely
supported probability measure on $\Aff(K)^2$, supported on pairs with a common
linear part and satisfying assumptions (1) and (2) of Theorem \ref{thm.correlated-refined}. 
Then there exist a finitely generated subring $R\subset K$ containing $A$ and the  entries of $\langle \supp \eta\rangle_+$, a global field
$F$, and a ring homomorphism
$
\theta:R\rightarrow F
$
such that 
    $\theta_{|_A}$ is injective
and  the induced 
group homomorphism 
$\Theta:\Aff(R)^2\to \Aff(F)^2$ is such that the measure $\Theta_\ast \eta$ on $\Aff(F)^2$ is still supported on pairs with a common linear parts and satisfies the assumptions (1) and (2) of $\eta. $
 \end{lemma}
\begin{proof}
By assumption (1),   there exists $h=(h_1,h_2)\in \supp \eta$ such that $a(h)\not\in \mu(K)$. By assumption (2), there exists $g=(g_1,g_2)\in \supp \eta$ such that $g_1\neq g_2$. Since $a(g_1)=a(g_2):=a$, writing  $$g_i(x)=ax+b_i, \qquad i=1,2,$$
one has $$u:=b_1-b_2\neq 0.$$
Let $K_0$ be the prime field of $K$ i.e. $K_0=\Q$ if $\textrm{char}(K)=0$ and $K_0=\mathbf{F}_q$ if $F$ is a finite extension of $\mathbf{F}_q$. Let $R\subset K$ be the finitely generated $\mathbf{F}_0$-algebra generated by 
\begin{itemize}
    \item all coefficients $a(g),b(g_1),b(g_2)$ with $g=(g_1,g_2)\in \supp \eta$, 
    \item the inverses of all linear parts $a(g)$, $g\in \supp \eta$, 
    \item the elements $u$ and $u^{-1}$. 
    \end{itemize}
    Note that $A\subset R^\times$. 
By
Grunewald--Segal~\cite[Theorem~A]{GS}, there exists a global field $F$ and a ring homomorphism $\theta: R\to F$ injective on $R^\times$. 
By construction $R$ contains the entries of $\langle \supp \eta\rangle_+$ so $\theta$ induces a ring homomorphism $\Theta: \Aff(R)^2\to \Aff(F)^2$. 
 Since the elements in the support of $\eta$ have the same linear   parts, the same is true after applying $\Theta$. Moreover, since $A$ is generated by $a(g)^{\pm 1}$ with $g\in \supp \eta$, $A\subset R^{\times}$ by construction.  Injectivity of $\theta$ on $R^\times$ implies that on $A$. Now we verify assumption (1) and (2) for $\Theta_\ast \eta$. By assumption (1) on $\eta$, there exists $g\in \supp \eta$ with $a(g)\not\in \mu(K)$. By injectivity of $\theta$ on $A$, this implies that $a(\Theta(g))=\theta(a(g))$ is also not a root of unity. Assumption (1) is then fulfilled for $\Theta_\ast \eta$. 
  Finally, $u\in R^\times$ by construction. By injectivity of $\theta$ on $R^{\times}$,  $\theta(u)\neq 0$. Therefore $\Theta(g_1)\neq \Theta(g_2)$ so $\supp(\Theta_\ast \eta)\not\subset \Delta$ and assumption (2) is satisfied by $\Theta_\ast \eta$.
 \end{proof}
 \begin{remark}

In general one cannot require $\theta$ to be injective on all of $R$. Indeed if some coefficient of an element in $\supp \eta$ is transcendental over $\Q$, say $t$, then any such $R$ contains $\Q[t]$.   Since any element of a number field is
algebraic, every homomorphism $R\to F$ to a number field annihilates a nonzero
polynomial in $t$ and therefore cannot be injective. The role of the Grunewald--Segal theorem is precisely to guarantee injectivity on a prescribed finitely generated subgroup of units, rather than on the whole coefficient ring.
\end{remark}
 
Now we are ready to prove Theorem  \ref{thm.correlated-refined}.

\begin{proof}[Proof of Theorem~\ref{thm.correlated-refined}]

Assume first that $(a(L_n))$ is recurrent. Then, by
Proposition~\ref{prop.rw-abelian}, almost surely
$
a(L_n)=1$ i.o. 
Hence, almost surely, infinitely often both $L_{n,1}$ and $L_{n,2}$ are
translations. In particular, they commute so that    
$
\langle L_{n,1},L_{n,2}\rangle_+$
is not free infinitely often.

\medskip 
\medskip 

We now reduce the remaining assertions to the case where $K$ is a global
field. Let $R, F, \theta, \Theta$ as in Lemma \ref{lemma.specialization}. 
Since $\theta$ is injective on $A$, the projected walks $(a(L_n))$ and
$(a(\Theta(L_n)))$ have the same recurrence/transience behavior. Finally, any
relation between $L_{n,1}$ and $L_{n,2}$ specializes to the same relation
between $\Theta(L_{n,1})$ and $\Theta(L_{n,2})$. Hence freeness of the specialized
semigroup implies freeness of the original one. 

We may therefore assume from now on that $K$ is a  global field.
By Proposition \ref{prop.eventual-separation}, almost surely, there exists $n_0\in \N$ such that 
for every $n\geq n_0$, $L_{n,1}$ and $L_{n,2}$ have distinct fixed points whenever $a(L_n)\neq 1$. 
 
\textit{Case 1. Suppose $(a(L_n))$ transient.} 
 By Proposition~\ref{prop.rw-abelian}, almost surely the walk
$(a(L_n))$ eventually leaves every finite subset of $A$.
Since by Proposition \ref{lemma.ker-lambda} the logarithmic embedding $\lambda$ is proper, it follows that
\[
\|\lambda(a(L_n))\|\longrightarrow+\infty,
\]
where $\|\cdot\|$ is any norm on $\R^{S(A)}$.
 Hence, almost surely there exists $n_1\in \N$ such that for every $n\geq n_1$, $|a(L_n)|_{v_n}<1/3$ or $|a(L_n)|_{v_n}>3$. Let $n\geq \max\{n_0,n_1\}$. Suppose first $|a(L_n)|_{v_n}<1/3$. Then  $L_{n,1}$ and $L_{n,2}$ are sufficiently contracting affine maps of $K_{v_n}$ with distinct fixed points. By the ping-pong lemma \ref{lemma.ping.pong},  almost surely and for every $n\geq \max\{n_0,n_1\}$,  $\langle L_{n,1}, L_{n,2}\rangle_+$ is free.  If instead $|a(L_n)|>3$, then the same argument applied to $L_n^{-1}$ yields the freeness of $\langle L_{n,1}^{-1}, L_{n,2}^{-1}\rangle_+$ and hence that of $\langle L_{n,1}, L_{n,2}\rangle_+$. In any case, a.s. for $n\geq \max\{n_0,n_1\}$, the semigroup generated by $L_{n,1}$ and $L_{n,2}$ is free. 

\textit{Case 2. Suppose $(a(L_n))$ is recurrent.}
By Lemma~\ref{lem.centered-semigroup-group}, the walk is irreducible on  $A$.  An irreducible recurrent Markov chain admits a unique (up to scaling) stationary measure (see for instance \cite[Theorems 1.7.5, 1.7.6]{norris}). Since the walk is a random walk on the group $A$,  the counting measure is stationary. Hence every stationary measure is proportional to the counting measure. But the latter  is of infinite mass because, by assumption (1), $A$ is infinite. Hence   there is no stationary probability measure on $A$ and the  recurrent chain is null recurrent. Hence   $\E(\tau)=+\infty$, where $\tau$ is the first time the chain revisits the identity. Consequently, by the ergodic theorem for Markov chains (see for instance \cite[Theorem 1.10.2]{norris}), for
every finite set $B\subset A$,
a.s., \[
\frac{1}{N}\#\{1\leq n\leq N:\ a(L_n)\in B\}
\underset{N\to+\infty}{\longrightarrow}0.
\]
Let
\[
B:=\lambda^{-1}\left(\prod_{v\in S(A)}[-\log 3,\log 3]\right).
\]
Since $\lambda$ is proper, $B$ is finite. Therefore, outside a set of times of
asymptotic density zero, one has $a(L_n)\notin B$. For such $n$, there exists
$v_n\in S(A)$ such that
\[
|a(L_n)|_{v_n}<1/2
\qquad\text{or}\qquad
|a(L_n)|_{v_n}>2.
\]
Using Proposition~\ref{prop.eventual-separation} and arguing as in the
transient case, Lemma~\ref{lemma.ping.pong} implies that
$
\langle L_{n,1},L_{n,2}\rangle_+$
is free for all such sufficiently large $n$. Hence the set of non-free times
has asymptotic density zero. This completes the proof.
\end{proof}

\begin{lemma}[Ping--pong for affine contractions]\label{lemma.ping.pong}
Let $\KK$ be a local field and let $g_1,g_2\in\Aff(\KK)$ have the 
same linear part $a$ with $|a|_\KK<\frac{1}{3}$ and distinct fixed 
points. Then $\langle g_1,g_2\rangle_+$ is free.
\end{lemma}

\begin{proof}
Let $x_1^+,x_2^+$ be the fixed points and set $d=|x_1^+-x_2^+|$ 
and $U_i:=B(x_i^+,d/2)$. The balls $U_1,U_2$ are disjoint. Moreover, since  
$g_i(x)=a(x-x_i^+)+x_i^+$, for $x\in U_j$, we  have $$|g_i(x)-x_i^+|=|a||x-x_i^+|\leq |a|(d/2+d)\leq \frac{3}{2}|a| d<d/2.$$ 
Hence $g_1(U_2)\subset U_1$ and  $g_2(U_1)\subset U_2$. 
The standard ping--pong 
lemma gives freeness.
\end{proof}
We deduce the
\begin{proof}[Proof of Corollary~\ref{cor.intro-criterion}]
By Theorem~\ref{thm.correlated-refined}, $\langle L_{n,1}, L_{n,2}\rangle_+$ is a.s.~eventually free  if and only if  
  the projected walk $(a(L_n))$ on the abelian group $A$ is transient.
  If the abelian rank $r$ of $A$ is $\geq 3$, then  Proposition~\ref{prop.rw-abelian} guarantees  the projected walk on $A$ is transient,  regardless whether the driving measure $a_\ast \eta$ is centered or not; thus    eventual freeness  holds a.s.. Conversely, if  $r\in \{1,2\}$, Proposition \ref{prop.rw-abelian}
  shows that $(a(L_n))$   is transient if and only if  $a_\ast \eta$ is not centered, which yields the desired dichotomy.
\end{proof}

\begin{remark}[Arithmetic sharpening]\label{rem.littlewood.sharp}
The proof above identifies the obstruction to freeness as the 
finite set $B=\{a\in A:|a|_v\in[1/3,3]\ \forall v\}$. We now 
show that the exact obstruction admits a sharper description.
By Proposition~\ref{prop.eventual-separation}, almost surely 
there exists $n_0\in\N$ such that for all $n\ge n_0$ with 
$a(L_n)\notin\mu(K)$, the fixed points of $L_{n,1}$ and 
$L_{n,2}$ are distinct. For such $n$, conjugating, we may write 
$L_{n,1}(x)=ax$ and $L_{n,2}(x)=ax+t$ with $t\neq 0$ and 
$a=a(L_n)$. An induction on the length of the word shows that 
a word $w=w_0\cdots w_{k-1}$ on $L_{n,1},L_{n,2}$ satisfies
\[
w(x)=a^k x + t P_w(a),
\]
where $P_w=\sum_{i=0}^{k-1}\varepsilon_i X^i$ with 
$\varepsilon_i\in\{0,1\}$. Two distinct words of the same length 
give the same map if and only if their difference $P_w - P_{w'}$ 
is a non-zero polynomial with coefficients in $\{-1,0,1\}$ 
vanishing at $a$; we call such a polynomial a \textit{Littlewood 
polynomial}. Conversely, 
any Littlewood polynomial vanishing at $a$ yields two distinct 
words with the same image. Hence, for all large $n$ with 
$a(L_n)\notin\mu(K)$, the semigroup $\langle L_{n,1},L_{n,2}
\rangle_+$ fails to be free if and only if $a(L_n)$ is a root 
of a Littlewood polynomial.
Since every root $z$ of a Littlewood polynomial satisfies 
$|z|_v\in[\frac{1}{2},2]$ at every place $v$, such roots lie 
inside $B$ --- they are precisely what the ping--pong argument 
cannot detect. Together with roots of unity, which obstruct 
freeness independently of separation, this gives: almost surely 
for all large $n$,
\[
\langle L_{n,1},L_{n,2}\rangle_+\text{ is not free}
\Longleftrightarrow
a(L_n)\in\mu(K)\ \text{or}\ a(L_n)\ \text{is a root of a 
Littlewood polynomial.}
\]
\end{remark}

 \appendix
\section{Local contraction for affine random similarities  over local fields }
The goal of this appendix is to prove a local contraction theorem for centered (or critical) random walks
on the group $\ASim(V)$ of affine similarities of a vector space $V$ over an arbitrary local field $\KK$. We prove this result under natural moment conditions, which are more general than the ones used in the paper. 
It unifies the Archimedean case of
Babillot--Bougerol--\'Elie~\cite{babillot.bougerol.elie}
with the non-Archimedean analog initiated by
Cartwright--Kaimanovich--Woess~\cite{CKW}
and further developed by Brofferio~\cite{brofferio.thesis,brofferio.renewal} (in the more general framework of the affine group of a homogeneous tree).  The case $\KK=\C$
  does not appear to have been treated previously. It also extends previous results to the natural generality of affine similarities rather than scalar dilations, which is the appropriate class here since finite extensions of local fields naturally give rise to linear similarities which need not be scalar dilations. 
The local field generality is needed in the proof of Proposition \ref{prop.eventual-separation} and in the proof of Theorem \ref{thm.correlated-refined}. 

The contraction ratio $c(g)\in\R_{>0}$ remains 
scalar throughout, keeping the setting one-dimensional in spirit. We follow  Brofferio's alternative proof in \cite{brofferio-how} of Babillot--Bougerol--\'Elie's result \cite{babillot.bougerol.elie}. 
 
\medskip\noindent\textbf{Setup.}
Let $\KK$ be a local field (i.e. $\KK=\R,\C$, a finite extension of $\Q_p$ for some prime $p$, or a finite extension of the field of formal Laurent series  $\mathbf{F}_q((T))$ over a finite field $\mathbf{F}_q$; note that in the latter case the characteristic of the field is positive). Let $V$ be a finite-dimensional $\KK$-vector space
equipped with a norm denoted by $\|\cdot\|$. Define the group of linear similarities by
\[
\Sim_{\lin}(V):=\{T\in\GL(V):\|Tx\|=c(T)\|x\|\ \forall x\in V\}
\]
and the group of affine similarities by
$\ASim(V):=V\rtimes\Sim_{\lin}(V)$,
with elements $(b,T)$ acting by $x\mapsto Tx+b$ and group law
$(b,T)(b',T')=(b+Tb',TT')$.

\begin{remark}
Over $\R$ or $\C$, every similarity is a scalar dilation composed with an
isometry, so $\Sim_{\lin}(V)=\R_{>0}\cdot\Isom(V)$. Over non-archimedean
fields this equality may fail. For instance, in a totally ramified extension
$V=\Q_p(\sqrt[d]{p})$ of $\Q_p$, viewed as a $\Q_p$-vector space, multiplication by
the uniformizer $\sqrt[d]{p}$ is a $\Q_p$-linear similarity of ratio
$p^{-1/d}$ but is not in $\Q_p^\ast\cdot\Isom(V)$ when $d>1$: indeed,
its eigenvalues have modulus $p^{-1/d}\notin p^{\Z}=|\Q_p^\ast|$.
\end{remark}

\begin{theorem}[Local contraction]
\label{thm.local-contraction}
Let $\mu$ be a probability measure on $\ASim(V)$ satisfying:
\begin{enumerate}
\item[\rm(i)] \textit{(Moment condition)}
$\displaystyle\int\!\bigl(\log^+c(g)+\log^+\|b(g)\|\bigr)\,d\mu(g)<\infty$.
\item[\rm(ii)] \textit{(Affine non-degeneracy)}
There is no $v\in V$   fixed by all $g\in\supp\mu$.
\item[\rm(iii)] \textit{(Centered ratio)}
$\displaystyle\int\!\log c(g)\,d\mu(g)=0$.
\item[\rm(iv)] \textit{(Non-trivial ratio)}
$\mu\{g; c(g)=1\}<1$. 
\end{enumerate}
Let $L_n:=X_n\cdots X_1$ be the associated left random walk. Then for
every compact $U\subset V$ and every $x,y\in V$,
\[
\|L_nx-L_ny\|\,\mathbf{1}_{\{L_nx\in U\}}\longrightarrow 0
\qquad\text{almost surely}.
\]
\end{theorem}
\begin{remark}
The local contraction statement is non-vacuous only when the walk 
returns to compact sets infinitely often, i.e.\ when it is 
topologically recurrent. When $\KK_0=\R$, topological recurrence 
of centered affine random walks on $\ASim(V)$ has been established 
in~\cite{aoun-brofferio-peigne}. The present theorem applies 
regardless, since its proof does not rely on recurrence but on 
total dissipativity of the Markov operator, which follows from 
topological transience of the walk on the group $G_\mu$ 
(Corollary~\ref{cor.transient}).
\end{remark}
\noindent\textbf{Reduction to $\KK_0=\R$, $\KK_0=\Q_p$ and $\KK_0=\mathbf{F}_q((T))$}.
Since $\|L_nx-L_ny\|=c(L_n)\|x-y\|$, it suffices to show
$c(L_n)\,\mathbf{1}_{\{L_nx\in U\}}\to0$ a.s.
Let $\KK$ be a local field. If $\KK$ is archimedean (i.e. $\R$ or $\C$), set $\KK_0=\R$. If $\KK$ is non-archimedean and has characteristic zero, set $\KK_0=\Q_p$ where $p$ is the characteristic of the residue field. If $\KK$ has positive characteristic, set $\KK_0=
\mathbf{F}_q((T))$ where $\mathbf{F}_q$ is the residue field. In all cases $\KK$ is a finite extension of $\KK_0$.
Viewing $V$ as a $\KK_0$-vector
space, any $\KK$-linear map is $\KK_0$-linear; in particular if
$T\in\Sim_{\lin}(V)$ satisfies $\|Tx\|=c(T)\|x\|$ for all $x\in V$,
this holds over $\KK_0$ with the same ratio $c(T)$. It is straightforward to verify that assumptions~(i)--(iv) are unchanged. It therefore suffices to treat $\KK_0=\R$, $\KK_0=\Q_p$ and $\KK_0=\mathbf{F}_q((T))$. Without loss of generality, when $\KK_0=\Q_p$ or $\KK_0=\mathbf{F}_q((T))$ we equip $V$ with an ultrametric norm (i.e.~$\|x+y\|\leq \max\{\|x\|, \|y\|\}$ for every $x,y\in V$) and, when $\KK_0=\R$, we equip $V$ with an Euclidean norm with respect to a fixed inner product on $V$. 
In all cases, $d$ denotes $\dim_{\KK_0}V$.
\medskip
The proof proceeds in three steps.

\subsection*{Step 1. Non-unimodularity and topological transience}

\begin{lemma}\label{lemma.non-unimodular}
Let $G$ be a closed subgroup of $\ASim(V)$ such that $c(G)\neq\{1\}$ and $G$ has no common fixed point in $V$. Then $G$ is non-unimodular.
\end{lemma}
\begin{proof}
Pick $g\in G$ with $c(g)<1$. Up to conjugating $G$ inside $\ASim(V)$, we may assume that the translation part of $g$ is zero. Let
$A_0$ be the subgroup of $G$ consisting of transformations that fix the origin  and note that $g\in A_0$. Note that both the assumption and the conclusion are preserved under passing from $G$ to this conjugate of $G$.  Arguing by contradiction, assume that   $G$ unimodular  and denote by $m$ its Haar measure. 

Let $A\subset G$ be compact.  For any $h=(b,T)\in A$,
\[
g^n h g^{-n} = (T_g^n b,\; T_g^n T T_g^{-n}).
\]
Since $c(g)^n\to0$, the translation part  $T_g^n b$ of $g^n h g^{-n}$ tends  to zero uniformly over the compact set $A$. In addition, since $c(T_g^nTT_g^{-n})=c(T)$,
the linear part of $g^n h g^{-n}$ stays in a compact subset of $\Sim_{\lin}(V)$.  Hence for every open neighborhood $U$ of $A_0$, there exists $N$ such that for every $n\geq N$, $g^n A g^{-n}\subset U$. Hence 
$m(g^n A g^{-n})\leq m(U)$. 
But by unimodularity of $G$, $m(g^n A g^{-n})=m(A)$. Hence $m(A)\leq m(U)$ for every open neighborhood $U$ of $A_0$ so that, by outer regularity of $m$, $m(A)\leq m(A_0)$. 

If $m(A_0)>0$, Steinhaus  theorem implies that the group  $A_0$ is open. By~(ii), pick $h\in G$ with
$b(h)\neq0$. Then $g^{n_k}hg^{-n_k}\to\gamma\in A_0$. Since $A_0$ is open in $G$, for large enough $k$, we have
$g^{n_k}hg^{-n_k}\in A_0$, i.e.\ $h\in g^{-n_k}A_0g^{n_k}=A_0$,
contradicting $b(h)\neq0$. So $m(A_0)=0$, hence  $m(A)=0$ for any compact subset $A\subset G$ so $m\equiv0$, absurd.
\end{proof}

\begin{corollary}\label{cor.transient}
Under assumptions \textit{(ii)} and \textit{(iv)}, the closed group
$G_\mu:=\overline{\langle\supp\mu\rangle}$ is non-unimodular, and the
$\mu$-random walk on $G_\mu$ is topologically transient: for every compact
$K\subset G_\mu$, a.s.\ $L_n\notin K$ for all large $n$.
\end{corollary}

\begin{proof}
Non-unimodularity follows from Lemma~\ref{lemma.non-unimodular} applied to
$G=G_\mu$: assumption~(iv) gives $c(G_\mu)\neq\{1\}$ and~(ii) gives no
common fixed point. Topological transience then follows from the theorem of
Guivarc'h--Keane--Roynette~\cite[Th\'{e}or\`{e}me 51]{GKR},
see also Guivarc'h--Raja \cite[Proposition 2]{guivarch-raja}.
\end{proof}
\subsection*{Step 2. Escape from a cusp, for a.e.\ starting point}

Define
\[
C:=\{(b,T)\in\ASim(V):c(T)\geq1,\;\|b\|\leq c(T)\}.
\]
This is an enlarged cusp: the natural cusp $\{c(T)>1,\,\|b\|\leq1\}$ is
contained in $C$, but the coupled condition $\|b\|\leq c(T)$ makes $C$
stable under right multiplication by elements with $c(T)>1$, which is
the key to upgrading from $m$-a.e.\ to every starting point in Step~3.

The strategy is to show that $L_ng$ eventually leaves $C$ for $m$-a.e.\
$g$, using total dissipativity of the Markov operator.
Although $L_n$ is supported on  $G_\mu$ it acts by left multiplication on 
$\ASim(V)$, preserving $m$. Moreover topological transience on $G_\mu$ implies that on $\ASim$. 
By Brofferio~\cite[Lemma~1]{brofferio-how}, topological transience of the
random walk on $G_\mu$ (Corollary~\ref{cor.transient}) implies that
$P_\mu$ is \textit{totally dissipative} on $L^1(\ASim(V),m)$: for every
$f\in L^1_+(\ASim(V),m)$ and $m$-a.e.\ $g$,
$\sum_{n=0}^\infty P_\mu^n f(g)<\infty$.
We apply this to the entrance function
$\phi(g):=\mathbf{1}_{C^c}(g)\,\p(X_1g\in C)$,
whose potential counts the expected number of crossings from $C^c$ into $C$.
The key point of the proof to to show that is $\phi\in L^1$. We will start by two preliminary lemmas.
\begin{lemma}[Haar measure on $\ASim(V)$]\label{lemma.haar}
The group $K:=\ker c=\Isom(V)$ is compact, and the sequence
\[
1\longrightarrow K\longrightarrow\Sim_{\lin}(V)\xrightarrow{\;c\;}\Lambda\longrightarrow1
\]
is exact, where $\Lambda:=c(\Sim_{\lin}(V))$. For any measurable section
$s:\Lambda\to\Sim_{\lin}(V)$, the map $(r,k)\mapsto s(r)k$ is a measurable
bijection $\Lambda\times K\to\Sim_{\lin}(V)$, and $dm_{\lin}$ disintegrates as
\[
dm_{\lin}(T)=dk\,d\nu(r),\qquad r=c(T),\quad k=s(r)^{-1}T,
\]
independently of the choice of $s$. Consequently,
\[
dm(b,T)=c(T)^{-d}\,db\,dk\,d\nu(r)
\]
is a left Haar measure on $\ASim(V)$, where $dk$ is the Haar probability measure
on $K$, and $d\nu$ is a Haar measure on the abelian group $\Lambda$.

\end{lemma}

\begin{proof}
$K$ is compact since isometries form a bounded closed subgroup of $\GL(V)$
over a local field. Since all eigenvalues of $T\in\Sim_{\lin}(V)$ have
modulus $c(T)$, we have $|\det T|_{K_0}=c(T)^d$, hence $T_*db=c(T)^d\,db$.
Left invariance of $c(T)^{-d}\,db\,dm_{\lin}$ follows by the semidirect
product formula. Left invariance of $dk\,d\nu(r)$: left multiplication by
$T_0=s(r_0)k_0$ maps $(r,k)\mapsto(r_0r,k')$ for some $k'\in K$, so
invariance follows from left invariance of $d\nu$ on $\Lambda$ and of $dk$
on $K$. Independence of $s$ follows from left invariance of $dk$.
\end{proof}
Note that  when $\KK_0=\R$, $\Lambda=\R_+^*$ and
$d\nu(r)=dr/r$. When $\KK_0$ is non-archimedean with $q$ the cardinality of the residue field, since $c(T)^d=|\det T|_{\KK_0}\in q^\Z$,
$\Lambda\subset q^{\frac{1}{d}\Z}$ is discrete and $d\nu$ is counting
measure on $\Lambda$.
\begin{lemma}[Moon lemma]\label{lemma.moon}
There exists $C_0>0$ such that for every $B\in V$ and $t>0$,
\[
\mathrm{Vol}\bigl(\{u\in V:\|u\|>t,\;\|u+B\|\leq t\}\bigr)
\leq C_0\min(t^d,\,\|B\|\,t^{d-1}).
\]
\end{lemma}

\begin{proof}
Set
\[ S:= \{u\in V:\|u\|>t,\;\|u+B\|\leq t\}. \] 
In order to estimate the volume of $S$, we will distinguish two cases. 
\medskip

\noindent\textit{Case 1:  $\KK_0$ non-archimedean}.

Observe that $S=\mathbf{B}(-B,t)\setminus \mathbf{B}(0,t)$, where $\mathbf{B}(a,r)$ denotes the closed ball of center $a\in V$ and radius $r$. Now in a ultrametric metric space, two balls of the same radius are either disjoint or equal. Hence  the two balls are equal when 
  $\|B\|\leq t$ (since then $-B\in \mathbf{B}(0,t)$) and disjoint otherwise. In other terms, $S=\emptyset$ if $\|B\|\leq t$ or $S=\mathbf{B}(-B,t)$ if $\|B\|>t$. In the first case, $\textrm{Vol}(S)=0$ and the desired inequality is trivial. If $\|B\|>t$, 
  $\textrm{Vol}(S)=\textrm{Vol}(\mathbf{B}(-B,t))\leq C_0 t^d=C_0 t \,t^{d-1}\leq C_0 \|B\| t^{d-1}$. 

\medskip

\noindent\textit{Case 2: $\KK_0=\R$.}
Decompose $V=\R u_0\oplus u_0^\perp$ with $u_0=B/\|B\|$. Writing
$u=su_0+z$ with $s\in\R$ and $z\in u_0^\perp$, we have 
$\|u\|^2=s^2+\|z\|^2$ and $\|u+B\|^2=(s+\|B\|)^2+\|z\|^2$.
The condition $\|u+B\|\leq t$ forces $\|z\|\leq t$. For fixed such $z$,
set $a(z):=\sqrt{t^2-\|z\|^2}$; then
\[
\|u\|>t\iff |s|>a(z),\qquad
\|u+B\|\leq t\iff |s+\|B\||\leq a(z).
\]
Hence the slice of $S$ above $z$ is
\[
I_z=\{s\in\R:|s|>a(z),\;|s+\|B\||\leq a(z)\}
=[-\|B\|-a(z),\,-\|B\|+a(z)]\setminus[-a(z),\,a(z)],
\]
the set difference of two intervals of the same length $2a(z)$, one being
the translate of the other by $-\|B\|$. When $\|B\|\geq 2a(z)$ they
do not intersect and $\ell(I_z)=2a(z)\leq\|B\|$; when $\|B\|<2a(z)$
they intersect and $\ell(I_z)=\|B\|$. Hence $\ell(I_z)\leq\min(2a(z),\|B\|)$.
By Fubini,
\[
\mathrm{Vol}(S)
\leq\int_{\|z\|\leq t}\ell(I_z)\,dz
\leq\int_{\|z\|\leq t}\min(2a(z),\|B\|)\,dz.
\]
Since $\min(2a(z),\|B\|)\leq 2a(z)\leq 2t$, the first bound gives
$\mathrm{Vol}(S)\leq Ct^d$. Since $\min(2a(z),\|B\|)\leq\|B\|$ and
$\mathrm{Vol}(\{z\in u_0^\perp:\|z\|\leq t\})=Ct^{d-1}$, the second
gives $\mathrm{Vol}(S)\leq C\|B\|t^{d-1}$. Hence
$\mathrm{Vol}(S)\leq C\min(t^d,\|B\|t^{d-1})$.
\end{proof}

\begin{lemma}\label{lemma.phi-L1}
The entrance function $\phi$ belongs to $L^1(\ASim(V),m)$.
\end{lemma}

 \begin{proof}
Decompose $C^c=D_1\sqcup D_2$ with $D_1=\{c(T)<1\}$ and
$D_2=\{c(T)\geq1,\,\|b\|>c(T)\}$, and write $X_1=(B_1,A_1)$.
By Lemma~\ref{lemma.haar} and Tonelli's theorem,   since $\phi(g)=\E[\mathbf{1}_{C^c}(g)\mathbf{1}_C(X_1g)]
$, 
\[
\int_{\ASim(V)}\phi(g)\,dm(g)
=\E\!\left[\int_{\ASim(V)}\mathbf{1}_{C^c}(g)\,\mathbf{1}_{C}(X_1g)\,dm(g)\right]
=I_1+I_2,
\]
where $I_j:=\E\!\left[\int_{D_j}\mathbf{1}_C(X_1g)\,dm(g)\right]$
and $dm(b,T)=r^{-d}\,db\,dk\,d\nu(r)$ with $r=c(T)$, $T=s(r)k$.
 
\medskip\noindent\textit{Integral $I_1$ over $D_1$.}
For $g=(b,s(r)k)\in D_1$ with $r<1$, we have $X_1g=(B_1+A_1b,\,A_1s(r)k)$,
so the condition $X_1g\in C$ reads $c(A_1s(r)k)\geq1$ and
$\|B_1+A_1b\|\leq c(A_1s(r)k)$. Since $c$ is a homomorphism and
$k\in K=\ker c$, we have $c(A_1s(r)k)=c(A_1)c(s(r))=c(A_1)r$.
In particular the integrand does not depend on $k$, so $\int_K dk=1$
gives
\[
I_1=\E\!\left[\int_{\{r<1\}}\int_V
\mathbf{1}_{\{c(A_1)r\geq1\}}\,
\mathbf{1}_{\{\|B_1+A_1b\|\leq c(A_1)r\}}\,
r^{-d}\,db\,d\nu(r)\right].
\]
By translation invariance of $db$,
$\int_V\mathbf{1}_{\{\|B_1+A_1b\|\leq c(A_1)r\}}\,db
=\int_V\mathbf{1}_{\{\|A_1b\|\leq c(A_1)r\}}\,db$.
Since $A_1$ is a similarity of ratio $c(A_1)$, the condition
$\|A_1b\|\leq c(A_1)r$ is equivalent to $\|b\|\leq r$, hence
\[
\int_V\mathbf{1}_{\{\|A_1b\|\leq c(A_1)r\}}\,db
=\int_V\mathbf{1}_{\{\|b\|\leq r\}}\,db=r^d\cdot\mathrm{Vol}(B(0,1)).
\]
The factors $r^{-d}$ and $r^d$ cancel, leaving
\[
I_1= C\,\E\bigl[\nu\bigl(\{r\in\Lambda:c(A_1)^{-1}\leq r<1\}\bigr)\bigr].
\]
When $\KK_0=\R$ we have the estimate $\nu(\{r:c(A_1)^{-1}\leq r<1\})=\int_{c(A_1)^{-1}}^1\frac{dr}{r}=\log^+c(A_1)$.

 \noindent When $\KK_0$ is non-archimedean, we have a similar  estimate:
$$\nu(\{r\in\Lambda:c(A_1)^{-1}\leq r<1\})=\#(\Lambda\cap[c(A_1)^{-1},1))
\leq C_p\log^+c(A_1),$$ 
where the last inequality is based on the fact that $\Lambda\subset q^{\frac{1}{d}\Z}$ is a geometric progression.
In both cases $I_1\leq C\,\E[\log^+c(A_1)]<\infty$ by assumption~(i).
 
\medskip\noindent\textit{Integral $I_2$ over $D_2$.}
For $g=(b,s(r)k)\in D_2$ with $r\geq1$ and $\|b\|>r$, the condition
$X_1g\in C$ reads $c(A_1)r\geq1$ and $\|B_1+A_1b\|\leq c(A_1)r$.
As in $I_1$, the integrand does not depend on $k$, so $\int_K dk=1$ and
\[
I_2=\E\!\left[\int_{\{r\geq1\}}\int_{\{\|b\|>r\}}
\mathbf{1}_{\{c(A_1)r\geq1\}}\,
\mathbf{1}_{\{\|B_1+A_1b\|\leq c(A_1)r\}}\,
r^{-d}\,db\,d\nu(r)\right].
\]
Set $R:=c(A_1)$ and $Q:=\|B_1\|$. Change variables $u=A_1b$; since
$A_1$ is a similarity of ratio $R$, $\|u\|=R\|b\|$ and $db=R^{-d}\,du$,
so the domain $\{\|b\|>r,\,\|B_1+A_1b\|\leq Rr\}$ becomes
\[
E_r:=\{u\in V:\|u\|>Rr,\;\|B_1+u\|\leq Rr\}.
\]
Hence
\[
I_2=\E\!\left[\int_{\{r\geq1,\,Rr\geq1\}}
R^{-d}\,r^{-d}\,\mathrm{Vol}(E_r)\,d\nu(r)\right].
\]
By the Moon lemma, $\mathrm{Vol}(E_r)\leq C_0\min((Rr)^d,Q(Rr)^{d-1})$, hence
\[
R^{-d}\,r^{-d}\,\mathrm{Vol}(E_r)\leq C_0\min\!\left(1,\frac{Q}{Rr}\right).
\]
Since $R=c(A_1)\in\Lambda$, left-invariance of $d\nu$ on $\Lambda$ gives
$d\nu(Rr)=d\nu(r)$, so the change of variables $t=Rr$ yields
\[
I_2\leq C\,\E\!\left[\int_{\{t\geq1\}}\min\!\left(1,\frac{Q}{t}\right)
d\nu(t)\right].
\]

\smallskip\noindent
When $\KK_0=\R$, $d\nu(t)=dt/t$ and splitting at $t=Q$:
\[
\int_1^\infty\min\!\left(1,\frac{Q}{t}\right)\frac{dt}{t}
=\int_1^Q\frac{dt}{t}+Q\int_Q^\infty\frac{dt}{t^2}
=\log^+Q+1.
\]

\smallskip\noindent
When $\KK_0$ is non-archimedean, $\Lambda=\{q^{na}:n\in\Z\}$ for some $a>0$,
$d\nu$ is counting measure, and splitting at $t=Q$ gives
\[
\sum_{t\in\Lambda,\,t\geq1}\min\!\left(1,\frac{Q}{t}\right)
=\#\{t\in\Lambda:1\leq t\leq Q\}+Q\sum_{t\in\Lambda,\,t>Q}\frac{1}{t}.
\]
The first term is $\leq C_a\log^+Q$. For the second, the tail of the
geometric series gives
$\sum_{t\in\Lambda,\,t>Q}1/t\leq C_a/Q$,
hence $Q\sum_{t\in\Lambda,\,t>Q}1/t\leq C_a$.
\smallskip\noindent
In both cases $I_2\leq C\,\E[1+\log^+\|B_1\|]<\infty$ by assumption~(i).
\end{proof}

\noindent\textbf{Conclusion of Step~2.}
Total dissipativity applied to $\phi\in L^1$ gives, for $m$-a.e.\ $g$,
\[
\sum_{n=0}^\infty \p(L_{n+1}g\in C,\;L_ng\notin C)<\infty.
\]
By Borel--Cantelli, $L_ng$ crosses from $C^c$ into $C$ only finitely many
times. By assumptions~(iii)--(iv) and~(i), the random walk $(\log c(L_n))$
has zero mean and finite first moment, so by the Chung--Fuchs
theorem~\cite{ChungFuchs1951}, $\liminf_{n}\log c(L_n)=-\infty$ a.s.,
i.e.\ $\liminf_{n}c(L_n)=0$ a.s. Since $c$ is a group homomorphism,
$\liminf_{n}c(L_ng)=0$ a.s., so $L_ng$ cannot remain trapped in $C$.
Hence for $m$-a.e.\ $g$, a.s.\ $L_ng\notin C$ for all large $n$.
\subsection*{Step 3. Upgrade from a.e.\ to every starting point}

\begin{lemma}\label{lemma.cusp-stable}
The set $E:=\{(u,S)\in\ASim(V):c(S)>1,\;\|u\|<c(S)-1\}$ has positive
Haar measure and satisfies $CE\subset C$.
\end{lemma}

\begin{proof}
Set $E$ is open and non-empty, hence has positive Haar measure. For $(b,T)\in C$ and
$(u,S)\in E$: the product is $(b+Tu,TS)$ with
$c(TS)=c(T)c(S)\geq c(S)>1$ and
\[
\|b+Tu\|\leq\|b\|+c(T)\|u\|
\leq c(T)+c(T)\|u\|
= c(T)(1+\|u\|)
< c(T)c(S)=c(TS),
\]
so $(b+Tu,TS)\in C$.
\end{proof}

\begin{corollary}\label{cor.escape-any-g}
For every $g_0\in\ASim(V)$, almost surely $L_ng_0\notin C$ for all large $n$.
\end{corollary}

\begin{proof}
Let $F\subset\ASim(V)$ be the full Haar measure set from Step~2.
Since $g_0^{-1}F$ has full measure and $E$ has positive measure,
choose $h\in E\cap g_0^{-1}F$. Since   $g_0h\in F$,
by Step~2, a.s.\ $L_n(g_0h)\notin C$ for large $n$. 
If $L_ng_0\in C$ infinitely often then  by Lemma~\ref{lemma.cusp-stable}, $L_ng_0h\in CE\subset C$
infinitely often, a contradiction.
\end{proof}

\subsection*{Conclusion of the proof}
We first bootstrap Corollary \ref{cor.escape-any-g} to a stronger statement that will be needed for the proof. 
For $M>0$, define
\[
C_M:=\{(b,T)\in\ASim(V):c(T)\geq1,\;\|b\|\leq Mc(T)\},
\]
so that $C_1=C$. 

We claim that the statement of Corollary \ref{cor.escape-any-g} continues to hold if we replace $C$ by $C_M$ in the conclusion.  Since conjugation by $(0,\lambda I_d)$ acts as
$(b,T)\mapsto(\lambda b,T)$, choosing $\lambda\in\KK_0$ with
$|\lambda|\leq M^{-1}$ gives
$(0,\lambda I_d)\,C_M\,(0,\lambda I_d)^{-1}\subset C_1$.
The conjugated measure $\mu^{(\lambda)}$ satisfies the same
assumptions~(i)--(iv). For any $g_0\in\ASim(V)$, setting
$g_0^{(\lambda)}:=(0,\lambda I_d)g_0(0,\lambda I_d)^{-1}$, the associated random walk $L_n^{(\lambda)}$ verifies 
$L_n^{(\lambda)}g_0^{(\lambda)}=(0,\lambda I_d)L_ng_0(0,\lambda I_d)^{-1}$. 
Hence, $L_ng_0\in C_M$ implies $L_n^{(\lambda)}g_0^{(\lambda)}\in C_1$.
Applying Corollary~\ref{cor.escape-any-g} to $\mu^{(\lambda)}$ and
$g_0^{(\lambda)}$ we have proven the following. 

\begin{corollary}\label{cor.escape-CM}
For every $M>0$ and every $g_0\in\ASim(V)$, almost surely
$L_ng_0\notin C_M$ for all large $n$.
\end{corollary}

Let us come back to the proof of the main claim. 
We show $c(L_n)\,\mathbf{1}_{\{L_nx\in U\}}\to0$ a.s.\ for every compact
$U\subset V$ and every $x\in V$.
Now fix $x\in V$, set $g_x:=(x,\Id)$, so $b(L_ng_x)=L_nx$ and
$c(L_ng_x)=c(L_n)$. Choose $M>0$ with $U\subset\{v:\|v\|\leq M\}$.
By Corollary~\ref{cor.escape-CM}, a.s.\ eventually $L_ng_x\notin C_M$.
If $L_nx\in U$ and $c(L_n)\geq1$, then $c(L_ng_x)\geq1$ and
$\|b(L_ng_x)\|=\|L_nx\|\leq M\leq Mc(L_n)=M \, c(L_ng_x)$,
so $L_ng_x\in C_M$, which is impossible for large $n$. Hence a.s.\ for large $n$,
\[
L_nx\in U\implies c(L_n)<1.
\]
Now let $\varepsilon>0$. The set
\[
K_{\varepsilon,M}:=\{(b,T)\in\ASim(V):\varepsilon\leq c(T)\leq1,\;\|b\|\leq M\}
\]
is compact because both the operator norm of the linear part is bounded away from zero and infinity and the translation part is bounded in norm in a local field).
Since $L_ng_x\in K_{\varepsilon,M}$ implies $L_n\in K_{\varepsilon,M}g_x^{-1}$,
which is compact, topological transience (Corollary~\ref{cor.transient})
gives $L_n\in K_{\varepsilon,M}g_x^{-1}$ only finitely often. Hence
a.s.\ for large $n$,
\[
L_nx\in U\implies c(L_n)<\varepsilon.
\]
Since $\varepsilon>0$ is arbitrary, $c(L_n)\,\mathbf{1}_{\{L_nx\in U\}}\to0$
a.s., and therefore
\[
\|L_nx-L_ny\|\,\mathbf{1}_{\{L_nx\in U\}}
=c(L_n)\|x-y\|\,\mathbf{1}_{\{L_nx\in U\}}\longrightarrow0.\qed
\]
\section*{Acknowledgements}
The authors would like to thank Çağrı Sert for encouraging this collaboration and for several helpful discussions.
They are also grateful to Sara Brofferio for several enlightening discussions concerning affine random walks.


\end{document}